\documentclass[12pt]{amsart}
\usepackage[english]{babel}
\usepackage{graphicx}
\usepackage{amsmath}
\usepackage{amsfonts}
\usepackage{amssymb}
\usepackage{amsthm}
\usepackage{hyperref}
\usepackage[T1]{fontenc}
\usepackage{color}
\usepackage{mathtools}
\usepackage{comment}
\usepackage{enumerate}

\input xy
\xyoption{all}

\newtheorem{thm}{Theorem}[section]

\newtheorem{lem}[thm]{Lemma}
\newtheorem{pro}[thm]{Proposition}

\theoremstyle{definition}
\newtheorem{defi}[thm]{Definition}

\newtheoremstyle{remarque}{}{}{}{}{\it}{.}{\newline}{}
\theoremstyle{remarque}
\newtheorem*{rem}{Remark}

\newcommand{\asd}[5]{% 
\setbox1=\hbox{\ensuremath{^{#1}}}% 
\setbox2=\hbox{\ensuremath{_{#2}}}% 
\setbox5=\hbox{\ensuremath{#5}}% 
\hspace{\ifnum\wd1>\wd2\wd1\else\wd2\fi}% 
\ensuremath{\copy5^{\hspace{-\wd1}\hspace{-\wd5}#1\hspace{\wd5}#3}% 
_{\hspace{-\wd2}\hspace{-\wd5}#2\hspace{\wd5}#4}% 
}}

\usepackage[OT2,T1]{fontenc}
\DeclareSymbolFont{cyrletters}{OT2}{wncyr}{m}{n}
\DeclareMathSymbol{\Sha}{\mathalpha}{cyrletters}{"58}
\DeclareMathSymbol{\Brusse}{\mathalpha}{cyrletters}{"42}

\newcommand{\F}{\mathbb{F}}
\newcommand{\id}{\mathrm{id}}
\renewcommand{\ker}{\mathrm{Ker}}

\renewcommand{\hom}{\mathrm{Hom}}
\newcommand{\ext}{\mathrm{Ext}}

\newcommand{\aut}{\mathrm{Aut}}
\newcommand{\out}{\mathrm{Out}}
\renewcommand{\int}{\mathrm{Int}}

\newcommand{\coc}{\mathrm{coc}}

\usepackage{marvosym}

\title{On extensions of algebraic groups}
\author{Mathieu Florence}
\address{Mathieu Florence: Sorbonne Universit\'e, 4 place Jussieu, 75005 Paris, France.}
\email{mathieu.florence@imj-prg.fr}
\author{Giancarlo Lucchini Arteche}
\address{Giancarlo Lucchini Arteche: Departamento de Matem\'aticas, Facultad de Ciencias, Universidad de Chile, Las Palmeras 3425, \~Nu\~noa, Santiago, Chile.}
\email{luco@uchile.cl}
\thanks{The second author was partially supported by CONICYT via Fondecyt Grant 11170016 and PAI Grant 79170034.}

\date{}

\begin{document}

\maketitle

\begin{abstract}
    We extend to the context of algebraic groups a classical result on extensions of abstract groups relating the set of isomorphism classes of extensions of $G$ by $H$ with that of extensions of $G$ by the center $Z$ of $H$. The proof should be easily generalizable to other contexts. We also study the subset of classes of split extensions and give a quick application by proving a finiteness result on these sets over a finite field.\\
    \textbf{MSC classes:} primary 14L99, 20G15, secondary 18D35
\end{abstract}

\section{Introduction}
Let $G,H$ be abstract groups and let
\[1\to H\to E\to G\to 1,\]
be a group extension of $G$ by $H$. It is well known that the action of $E$ on $H$ by conjugation induces a morphism $\kappa:G\to\out(H)=\aut(H)/\int(H)$ that is called an \emph{outer action} of $G$ on $H$. One can then consider the set $\ext(G,H,\kappa)$ of isomorphism classes of extensions inducing the outer action $\kappa$. In the particular case where $H=A$ is an abelian group, an outer action becomes an action (since $\int(A)$ is trivial) and $\ext(G,A,\kappa)$ gets a natural group structure by means of the Baer sum. Another way of seeing this group structure is by noting that $\ext(G,A,\kappa)$ is naturally isomorphic to the cohomology group $H^2(G,A)$ (cf.~for instance \cite[Thm.~1.2.4]{NSW}).

In the case of a general group $H$, another classical result from group theory (but arguably less known) is the following: If $Z$ denotes the center of $H$, then an outer action $\kappa$ of $G$ on $H$ induces an action $\kappa_Z$ of $G$ on $Z$ and the set $\ext(G,Z,\kappa_Z)\simeq H^2(G,Z)$ acts \emph{simply transitively} on $\ext(G,H,\kappa)$ (cf.~for instance \cite[IV, Thm.~8.8]{MacLane}). Such a result can be easily generalized to a profinite group setting (for instance, in nonabelian Galois cohomology, cf. \cite{SpringerH2} or \cite{Borovoi93}), or even to a group-scheme setting by means of (nonabelian) Hochschild cohomology (cf.~\cite{DemarcheHochschild}). However, this last case follows the explicit cocycle approach from group cohomology and thus it can only take into account extensions \emph{that admit a scheme-theoretic section} $G\to E$.

Thus, if one wants to study general extensions of, say, algebraic groups or group schemes, one is bound to use a different approach. An attempt to do this for algebraic groups was done by the second author in \cite{GLA-Ext} when the group $G$ is finite, but he used an \emph{ad-hoc} cocycle approach that cannot be generalized to arbitrary algebraic groups.\\

In this paper, we study this situation for algebraic groups (that is, group schemes of finite type over a base field) with a point of view as general as possible. In particular, our proof should be applicable to group objects in other categories (for instance, group functors) without too much work. It is in this context that we (re)prove the classical result from group theory (for the notations, see section \ref{sec extensions}):

\begin{thm}\label{main thm intro}
Let $G,H$ be algebraic groups over a field $k$, let $Z$ denote the center of $H$ and let
\[1\to H\to E\to G\to 1,\]
be an extension. Then $G$ acts naturally on $Z$ and the group $\ext(G,Z,E)$ acts simply transitively on the set $\ext(G,H,E)$.
\end{thm}

By center, we mean the schematic center all throughout the article. The group $\ext(G,Z,E)$ is simply the set of extensions of $G$ by $Z$ inducing the same $G$-action as the extension $E$ by conjugation (we will recall its group structure in Section \ref{sec extensions}). The set $\ext(G,H,E)$ is the set of extensions ``inducing the same outer action'' as $E$. However, the classical definition of outer actions is not practical for generalizations and we want to avoid any \textit{ad-hoc} definitions.\\

A natural consequence of our main theorem is that we may reduce the study of extensions of algebraic groups to that of extensions by abelian algebraic groups, which are much easier to work with. As an example of this, we prove the finiteness of the sets $\ext(G,H,E)$ for algebraic groups over a finite field in Section \ref{section application}.

\subsection*{Acknowledgements}
The authors would like to thank Mikhail Borovoi and Michel Brion for their comments and suggestions.

\section{Preliminaries}

We recall here some basics on group extensions and fiber products that will be necessary in order to state and prove our main theorem. Our philosophy here is to manipulate extensions of nonabelian groups using mainly the following notions: fiber products, and taking quotients by normal subgroups.\\

Throughout this article, we will say that a morphism $\phi :G\to H$ of algebraic groups is \emph{surjective} if it is an epimorphism for the faithfully flat topology. In that case $\phi$ is faithfully flat and we have an exact sequence of algebraic groups
\[1\to\ker (\phi)\to G\to H\to 1.\]
As a general reference on algebraic groups, the reader may consult Milne's book \cite{MilneAG}.

\subsection{Fiber products}
Let us recall a basic result on fiber products that will be useful later. Let $G_1,G_2,H$ be algebraic groups over a field $k$ and consider morphisms $\phi_i:G_i\to H$ for $i=1,2$. Then we have the following result on the corresponding fiber product (which is defined via the classical universal property).

\begin{pro}\label{prop prod fib}
The fiber product $G_1\times_H G_2$ is the subgroup of $G_1\times G_2$ given by $(\phi_1\times\phi_2)^{-1}(\Delta(H))$, where $\Delta:H\to H\times H$ is the diagonal morphism. In particular, if $\phi_i$ is surjective for $i=1,2$, there is an exact sequence
\[1\to \ker(\phi_1)\times\ker(\phi_2)\to G_1\times_H G_2\to H\to 1.\]
\end{pro}

The proof is an easy exercise left to the reader.

\subsection{Extensions}\label{sec extensions}
Let $G,H$ be algebraic groups over a field $k$.

\begin{defi}
An extension $E$ of $G$ by $H$ is an exact sequence
\[1\to H\to E\to G\to 1.\]
A morphism of extensions is an isomorphism $\phi:E\to E'$ that fits into a commutative diagram
\[\xymatrix{
1 \ar[r] & H \ar[r] \ar@{=}[d] & E \ar[r] \ar[d]^\phi_\sim & G \ar[r] \ar@{=}[d] & 1 \\
1 \ar[r] & H \ar[r] & E' \ar[r] & G \ar[r] & 1. \\
}\]
An extension $E$ is said to be \emph{split} if there exists a morphism of algebraic groups $s:G\to E$ that splits $E\to G$.
\end{defi}

Extensions are functorial in the following sense. Let
\[1\to H\to E\to G\to 1\]
be an extension, and let $f:G' \to G$ be a morphism. We can then form the pullback diagram
\[ \xymatrix{ 1 \ar[r] &  H \ar[r] \ar@{=}[d]  & f^*(E)  \ar[r] \ar[d]  &  G' \ar[d]^f \ar[r] & 1 \\ 1 \ar[r] &  H \ar[r]   & E  \ar[r]   &  G \ar[r] & 1,}\]
where $f^*(E)$ is defined as the fiber product $ E \times_G G'$.

Consider now a morphism $f:H \to H'$. One would like to define the pushforward $f_*(E)$. This is not well defined in general. However, when the morphism $f$ is surjective and $\ker(f)$ is normal in $E$, one can define $f_*(E)$ as 
\[f_*(E):= [1\to H'\to E/\ker(f) \to G\to 1]. \]
Note that $\ker(f)$ is always normal in $E$ if it is characteristic in $H$.\\

Recall that if $A=H$ is \emph{abelian}, we may consider the set $\ext(G,A)$ of isomorphism classes of extensions of $G$ by $A$ inducing the same action of $G$ on $A$ (cf.~\cite[Exp.~17, App.~A]{SGA3}). It is an abelian group for the Baer sum. Let us recall its construction. Let
\[1\to A\to E_i \to G\to 1,\]
be extensions of $G$ by $A$ for $i=1,2$ inducing the same $G$-action on $A$. We may then consider the fiber product $E:=E_1\times_G E_2$, which fits naturally into an exact sequence (cf.~Proposition \ref{prop prod fib})
\[1\to A\times A\to E\to G\to 1.\]
Since the action of $G$ on $A$ induced by $E_1$ and $E_2$ is the same, it is easy to see that the sum and difference morphisms
\[+:A\times A\to A\qquad\text{and}\qquad -:A\times A\to A,\]
are $G$-equivariant and thus their respective kernels are normal in $E$. This means that we may consider the pushforwards $+_*(E)$ and $-_*(E)$, which we denote by $E_1+E_2$ and $E_1-E_2$. The classes of these extensions correspond respectively to the (Baer) sum $[E_1]+[E_2]$ and the difference $[E_1]-[E_2]$ in $\ext(G,A)$.

\subsection{Automorphisms of extensions}
Let $G,H$ be algebraic groups over a field $k$ and denote by $Z$ the (schematic) center of $H$. Since $Z$ is characteristic in $H$, it is normal in $E$. In particular, $E$  acts naturally on $Z$ by conjugation and this action clearly factors through $G$.

\begin{pro}\label{prop aut ext}
Consider an extension
\[1 \to H \to E \xrightarrow{\pi} G \to 1,\]
and the natural $G$-action on $Z$ induced by $E$. Denote by $Z_0^1(G,Z)$ the group of Hochschild 1-cocycles (cf. \cite[II, \S3.1]{DG}). We have a canonical group isomorphism
\begin{align*}
Z_0^1(G,Z) &\to \aut(E),\\
\phi &\mapsto \mu\circ(\phi\circ\pi,\id),
\end{align*}
where $\mu:E\times E\to E$ denotes the group multiplication. In particular, if $E$ is a central extension, we get $\hom (G,H)\xrightarrow{\sim}\aut (E)$.
\end{pro}

\begin{proof}
Let $f:E \to E$ be an automorphism of the extension $E$. Since $f$ induces the identity on both $H$ and $G$, it can be written as $f(x)=\phi(\overline x)x$, where $x\mapsto \overline x$ denotes the projection $\pi:E\to G$, for a certain morphism of $k$-schemes $\phi:G\to H$ such that $\phi(\bar e)=e$. We see then that $f=\mu\circ(\phi\circ\pi,\id)$.

We only have left to prove that $\phi$ is a Hochschild 1-cocycle with values in $Z$. For $x,y \in E$, we compute:
\[\phi(\overline{xy})xy=f(xy)=f(x)f(y)=\phi(\overline x)x\phi(\overline y)y,\] whence
\[\phi(\overline{xy})x=\phi(\overline x)x\phi(\overline y).\]
Taking $x$ in $H$, we see that $\phi(\bar y)$ commutes with $x$ and hence $\phi$ takes values in $Z$. This last relation also shows that $\phi: G \to Z$ is a $1$-cocycle for the action by conjugation. On the other hand, one easily checks that any $1$-cocycle defines an automorphism of $E$ via the formula above and, by associativity of $\mu$, that the sum of two cocycles maps to the composition of the corresponding automorphisms. The proposition is proved.
\end{proof}

\begin{rem}
This proof can of course be rewritten without the use of points, but it rapidly becomes cumbersome and it does not help its understanding. Moreover, this result is not used in the proof of our main theorem.
\end{rem}

\subsection{Outer actions}\label{sec out}
We will now define the notion of outer action (in a relative way) in terms of extensions. Let $G,H$ be algebraic groups, let $Z$ be the center of $H$ and let $E_1,E_2$ be extensions of $G$ by $H$. Since $Z$ is characteristic in $H$, it is normal in both $E_i$'s. We can thus consider, for $i=1,2$, the extensions 
\[1\to H/Z\to E_i/Z\to G\to 1.\]
Note that $E_i$ acts on $H$ by conjugation and that this action factors through $E_i/Z$, so that $E_i/Z$ acts naturally on $H$, hence on $Z$ as well.

\begin{defi}\label{defi action ext}
We say that $E_1$ and $E_2$ \emph{induce the same outer action} of $G$ on $H$ if there exists an isomorphism $\varphi$ between $E_1/Z$ and $E_2/Z$ as extensions of $G$ by $H/Z$, compatible with the natural actions on $H$. 

We define $\ext(G,H,E_1)$ as the set of isomorphism classes of extensions of $G$ by $H$ inducing the same outer action as $E_1$. Note that when $H=A$ is abelian, we recover the group $\ext(G,A)$ for the $G$-action obtained by conjugation in $E_1$.
\end{defi}

The advantage of this point of view is that it avoids the use of automorphism groups and replaces it with the notion of action, which is easier to define in a general setting (in particular in the context of algebraic groups). Indeed, in general automorphism groups are not group objects in the category one is working with and thus one has to give \textit{ad-hoc} definitions of outer actions in order to define the sets of extensions.

\begin{rem}
We have just replaced automorphisms by actions. If needed, these can be replaced by the notion of normal subgroups as follows. Fixing an isomorphism of extensions $E_1/Z\simeq E_2/Z$ amounts to giving a certain extension $E_0$ of $G$ by $H/Z$ and morphisms of extensions
\[\xymatrix{
1 \ar[r] & H \ar[r] \ar[d]_\pi & E_1 \ar[r] \ar[d] & G \ar[r] \ar@{=}[d] & 1 \\
1 \ar[r] & H/Z \ar[r] & E_0 \ar[r] & G \ar[r] & 1. \\
1 \ar[r] & H \ar[r] \ar[u]^\pi & E_2 \ar[r] \ar[u] & G \ar[r] \ar@{=}[u] & 1,
}\]
where $\pi:H\to H/Z$ is the natural projection (in particular, $E_0$ is unique up to isomorphism of extensions). Then we may consider the fiber product $E:=E_1\times_{E_0} E_2$, which fits into an exact sequence
\[1\to H\times_{H/Z}H\to E\to G\to 1,\]
as can be easily proved using Proposition \ref{prop prod fib}. If we consider the subgroup $\Delta(H)$ corresponding to the image of the diagonal embedding $\Delta:H\to H\times_{H/Z}H$, then we see the following.
\end{rem}

\begin{lem}\label{lem Delta H normal}
$E_1$ and $E_2$ induce the same action \emph{if and only if} $\Delta(H)$ is normal in $E$.
\end{lem}

Once again, this is an easy exercise using Proposition \ref{prop prod fib}. We leave the details to the reader.\\

In the classical setting of abstract groups, an outer action is defined as a morphism $\kappa:G\to \out(H)$ and an extension $E$ of $G$ by $H$ defines naturally an outer action via the commutative diagram
\[\xymatrix{
1 \ar[r] & H\ar[r] \ar@{->>}[d]_{\mathrm{can}} & E \ar[r] \ar[d]_{c} & G \ar[r] \ar[d]^\kappa & 1 \\
1\ar[r] & \int(H) \ar[r] & \aut(H) \ar[r]^\pi & \out(H) \ar[r] & 1,
}\]
where $c$ denotes conjugation in $E$. One can prove that in this context $E/Z$ corresponds to the fiber product $\aut(H)\times_{\out(H)} G$ and hence the image of $E/Z$ in $\aut(H)$ is completely determined by the datum of $\kappa:G\to\out(H)$. We deduce that two extensions having the same outer action in the classical setting will have the same image in $\aut(H)$, giving us an isomorphism $E_1/Z\simeq E_2/Z$ which evidently induces the same action.

On the other hand, if two extensions $E_1,E_2$ have isomorphic quotients by $Z$ and induce the same action on $H$, then clearly they both have the same image in $\aut(H)$ and thus induce the same morphism $\kappa:G\to\out(G)$ in the context of abstract groups.

\subsection{The action of $\ext(G,Z,E)$ on the set $\ext(G,H,E)$}\label{sec action of exts}
Let $G,H$ be algebraic groups and let $Z$ be the center of $H$ as above. Starting from an extension $E$ of $G$ by $H$, we get an action of $E/Z$ on $Z$ as we saw before, which factors through an action of $G$ on $Z$. Define $\ext(G,Z,E)$ to be the group of extensions of $G$ by $Z$ corresponding to this $G$-action. The purpose of this section is to define an action of $\ext(G,Z,E)$ on $\ext(G,H,E)$.\\

Let us consider then an extension
\[1\to Z\to E' \to G \to 1,\]
representing a class in $\ext(G,Z,E)$, i.e. inducing the same $G$-action on $Z$ as $E$. Consider the fiber product $E\times_G E'$. By Proposition \ref{prop prod fib}, this fits into an exact sequence
\begin{equation}\label{eqn def action}
1\to H\times Z \to E\times_G E'\to G \to 1.
\end{equation}
If we consider then the multiplication morphism $\mu:H\times Z\to H$, which is surjective and whose kernel is easily seen to be normal in $E\times_G E'$ (use for instance Proposition 2.1), we get by pushforward an extension
\[1\to H\to \mu_*(E\times_G E')\to G\to 1,\]
which induces the same outer action than $E$. Indeed, consider the extension
\[1\to H/Z\to E_0\to G\to 1,\]
obtained by pushing extension \eqref{eqn def action} via the natural arrow $H\times Z\to H/Z$. It is easy to see that there is a commutative diagram of extensions (where we omit the morphisms induced on the subgroups and quotients)
\[\xymatrix{ E\times_G E' \ar[r]^{\mu^*} \ar[d]_{p_1} & \mu_*(E\times_G E')  \ar[d]^{\pi_*} \\
E \ar[r]^{\pi_*} & E_0, 
}\]
where $\pi_*$ denotes the pushforward via $\pi:H\to H/Z$. Note that $H=H\times\{1\}$ is normal in $E\times_G E'$ (use for instance Proposition \ref{prop prod fib}) and that the action induced on $H$ by conjugation factors through $E\to E_0$ since the kernel of this morphism is $Z\times Z$, which clearly commutes with $H\times\{1\}$. We see then that the arrows $E\to E_0$ and $\mu_*(E\times_G E')\to E_0$ from the diagram above define an isomorphism inducing the same action on $H$ as needed.\\

One can check of course that this construction gives isomorphic extensions if we start with isomorphic extensions. We have thus defined a map
\[\begin{array}{ccl}
    \ext(G,H,E) \times \ext(G,Z,E) & \to & \ext(G,H,E),\\
    ([E],[E']) & \mapsto & [E]\cdot [E']:=[\mu_*(E\times_G E')],
\end{array}\]
which, when taking $H=Z$, recovers the classical group law on $\ext(G,Z,E)$, whose trivial element corresponds to the split extension $Z\rtimes G$. In general, we get an action of the group $\ext(G,Z,E)$ on the set $\ext(G,H,E)$. Checking this is an easy exercise using fiber products with three factors (and pushing forward via multiplication). We give some details in the proof of Theorem \ref{main thm} here below.

\section{Main Result}
We keep notations as in last section. In order to prove Theorem \ref{main thm intro}, we are only left to prove the following:

\begin{thm}\label{main thm}
Let $G,H$ be algebraic groups over a field $k$, let $Z$ denote the center of $G$ and let
\[1\to H\to E\to G\to 1,\]
be an extension. Then $\ext(G,Z,E)$ acts simply transitively on the set $\ext(G,H,E)$.
\end{thm}

\begin{proof}
Let $E_1,E_2$ be extensions of $G$ by $H$ inducing the same outer action as $E$. We need to find an extension $E'$ of $G$ by $Z$ inducing the same action as $E$ and such that $[E_1]\cdot [E']=[E_2]$.

By definition, both extensions come with a natural arrow $E_i\to E_0$, where $E_0$ is the extension of $G$ by $H/Z$ obtained by quotienting $Z\subset H\subset E_i$. Recalling Lemma \ref{lem Delta H normal}, we can consider the fiber product $E_{12}:=E_1\times_{E_0}E_2$, which fits into an exact sequence
\[1\to H\times_{H/Z} H\to E_{12}\to G\to 1,\]
and see that $\Delta(H)\subset H\times_{H/Z} H$ is normal in $E_{12}$. In particular, $\Delta(H)$ is normal in $H\times_{H/Z} H\subset H\times H$ and thus the composition
\[\nabla:H\times_{H/Z} H\xrightarrow{\iota\times\id} H\times H\xrightarrow{\mu} Z,\]
where $\iota$ denotes the inversion morphism, is a well-defined surjective morphism with kernel $\Delta(H)$. And knowing that $\Delta(H)$ is normal in $E_{12}$, we may consider the pushforward extension
\[1\to Z\to\nabla_*(E_{12})\to G\to 1.\]
Note that the action of $G$ on $Z$ induced by $\nabla_*(E_{12})$ is the same as the one induced by $E_1$ and $E_2$ since the restriction $\nabla:Z\times Z\to Z$ is clearly $G$-equivariant.

We claim now that $E':=\nabla_*(E_{12})$ is the extension we are looking for. Indeed, by construction, $E'$ fits in the commutative diagram:
\[\xymatrix{
E_{12} \ar[r] \ar[d] & E'=E_{12}/\Delta(H) \ar[d] \\
E_1 \ar[r] & G,
}\]
where the natural projection $E_{12}\to E_1$ also corresponds to the pushforward via the projection $p_1:H\times_{H/Z}H\to H$ on the first factor. Thus, we get a canonical morphism $\psi:E_{12}\to E_1\times_G E'$. Using Proposition \ref{prop prod fib} once again one checks that $\psi$ is in fact an isomorphism fitting in the following commutative diagram
\[\xymatrix{ 1 \ar[r] & H\times_{H/Z} H \ar[r] \ar[d]_{(p_1,\nabla)} & E_{12} \ar[r] \ar[d]_{\psi} & G \ar[r] \ar@{=}[d] & 1\\
1 \ar[r] & H\times Z \ar[r] & E_1\times_G E' \ar[r] & G \ar[r] & 1.
}\]
In particular, since $E'$ and $E_1$ induce the same action on $Z$, the pushforward of $E_{12}$ via the multiplication morphism $\mu:H\times Z\to Z$ is well defined, as we saw in section \ref{sec action of exts}. Moreover, by construction we can see that $m\circ (p_1,\nabla)=p_2$. This tells us that $\mu_*(p_1,\nabla)_*(E_{12})=\mu_*(E_1\times_G E')$ is canonically isomorphic to $E_2=(p_2)_*(E_{12})$ as an extension of $G$ by $H$, which means precisely that $[E_1]\cdot [E']=[E_2]$, as claimed. This proves the transitivity of the action.\\

Let us prove now the simple transitivity. In order to do this, we will simply show that, given an extension $E'$ of $G$ by $Z$ representing an element in $\ext(G,Z,E)$, the construction above applied to $E_1$ and $E_2:=\mu_*(E_1\times_G E')$ for $\mu:H\times Z\to Z$ gives back an extension isomorphic to $E'$.

Let us start then by noting that there is a natural map $E_1\times_G E'\to E_0$, which factors both through the projection $p_1:E_1\times_G E'\to E_1$ and $\mu_*:E_1\times_G E'\to E_2$. It is easy to see that there is a canonical isomorphism
\[E_1\times_{E_0}(E_1\times_G E')\simeq (E_1\times_{E_0} E_1)\times_G E'.\]
Indeed, by Proposition \ref{prop prod fib} both groups can be seen as the same subgroup of $E_1\times E_1\times E'$.

Since $E_1\times_G E'\to E_0$ factors through $\mu_*$, there is a natural pushforward morphism
\[\xymatrix{ 1 \ar[r] & H\times_{H/Z} H\times Z \ar[r] \ar[d]_{(\id,\mu)} & E_1\times_{E_0}E_1\times_G E' \ar[r] \ar[d]_{(\id,\mu_*)} & G \ar[r] \ar@{=}[d] & 1\\
1 \ar[r] & H\times_{H/Z} H \ar[r] & E_1\times_{E_0}E_2 \ar[r] & G \ar[r] & 1.
}\]
And thus the construction from the first part of the proof corresponds to the pushforward of this extension via $\nabla:H\times_{H/Z}H\to Z$, which is surjective and has kernel $\Delta(H)$ as we saw before.

On the other hand, it is easy to see (as we saw in the first part of the proof) that $\Delta(H)$ is normal in $E_1\times_{E_0} E_1$. This means that we may pushforward
\[1\to H\times_{H/Z}H\to E_1\times_{E_0} E_1\to G\to 1,\]
via $\nabla$ in order to induce a pushforward morphism
\[\xymatrix{ 1 \ar[r] & H\times_{H/Z} H\times Z \ar[r] \ar[d]_{(\nabla,\id)} & E_1\times_{E_0}E_1\times_G E' \ar[r] \ar[d]_{(\nabla_*,\id)} & G \ar[r] \ar@{=}[d] & 1\\
1 \ar[r] & H\times Z \ar[r] & \nabla_*(E_1\times_{E_0}E_1)\times_{G}E' \ar[r] & G \ar[r] & 1.
}\]
The key fact now is that we can pushforward this last extension via $\mu:H\times Z\to H$ (the reader can easily check that the kernel of $\mu$ is indeed normal in this last extension) and that the morphisms $\nabla\circ (\id,\mu)$ and $\mu\circ (\nabla,\id)$ are \emph{the same}. In particular, this pushforward is once again the one from the construction on the first part of the proof, but it represents at the same time the class \[[\nabla_*(E_1\times_{E_0}E_1)]+[E']\in\ext(G,Z,E).\]
However, it is easy to see that
\[E_1\times_{E_0}E_1=(Z\times\{1\})\rtimes\Delta(E_1)\cong Z\rtimes E_1,\]
where $\Delta$ denotes the diagonal embedding. Indeed, one can verify using Proposition \ref{prop prod fib} that $\Delta(E_1)\cap (Z\times\{1\})=\{1\}$, that $\Delta(E_1)$ normalizes $(Z\times\{1\})$ and that both generate $E_1\times_{E_0}E_1$. Thus, the kernel of $\nabla_*$ is identified with the subgroup $\{1\}\times H\subset Z\rtimes E_1$ and hence $\nabla_*(E_1\times_{E_0}E_1)$ is simply the split extension $Z\rtimes G$. We deduce that $[\nabla_*(E_1\times_{E_0}E_1)]+[E']=[E']$ and we get then an extension isomorphic to $E'$, which concludes the proof.
\end{proof}

\section{The set of split extensions}
Let $1\to H\to E\to G\to 1$ be an extension. Besides the structure of a principal homogeneous space under $\ext(G,Z,E)$, the set $\ext(G,H,E)$ has a distinguished subset corresponding to the classes of split extensions. We briefly study this set in this section.\\

Note that split extensions have by definition a scheme-theoretic section. Such extensions can be studied by using cocycles, so we use the cocyclic definition of nonabelian cohomology given by Demarche in \cite[\S2]{DemarcheHochschild}. Define then, for $H$ an algebraic group over a field $k$ with a $G$-group structure, the set $H^1_\coc(G,Z)$ as the quotient of the set of crossed morphisms
\[\{a:G\to H\mid a(gh)=a(g)\cdot {}^ga(h)\},\]
by the action of $H(k)$ sending $a$ to $(g\mapsto h\cdot a(g)\cdot {}^gh^{-1})$ for $h\in H(k)$.

Let us start with the existence of a (class of) split extensions in $\ext(G,H,E)$. Since all such extensions have the same quotient $E_0:=E/Z$, we immediately see that in order to have split extensions, we must have $E_0$ to be split as an extension of $G$ by $H/Z$. On the other hand, if we assume $E_0$ to be split, we may choose a section $s:G\to E_0$. Seeing $E$ as an extension of $E_0$ by $Z$ and pullbacking via $s$, we get the following commutative diagram
\begin{equation}\label{eq diag delta}
\xymatrix{
1 \ar[r] & Z \ar[r] \ar@{=}[d] & E_s \ar[r] \ar[d] & G \ar[r] \ar[d]^s & 1 \\
1 \ar[r] & Z \ar[r] \ar[d] & E \ar[r] \ar@{=}[d] & E_0 \ar[r] \ar[d] & 1 \\
1 \ar[r] & H \ar[r] & E \ar[r] & G \ar[r] & 1,
}
\end{equation}
which tells us that we can see $E_s$ as a subgroup of $E$. One easily checks that if we change $s$ by a conjugate section (over the base field), then we get an extension isomorphic to $E_s$. Now, if we fix a particular section $s_0:G\to E_0$, it induces a natural $G$-group structure on $H/Z$ and we know that sections $G\to E_0$ up to conjugacy are classified by $H^1_\coc(G,H/Z)$ for this $G$-group structure (cf.~\cite[Prop.~2.2.2]{DemarcheHochschild}). We have thus defined a map
\begin{align*}
    \delta:H^1_\coc(G,H/Z)&\to\ext(G,Z,E)\\
    \alpha=[s] &\mapsto [E_s],
\end{align*}
with which we may define a second map via the action on $\ext(G,H,E)$
\begin{align*}
    \delta_E:H^1_\coc(G,H/Z) &\to\ext(G,H,E)\\
    \alpha &\mapsto (-\delta(\alpha))\cdot [E].
\end{align*}
Note that if we change the choice of the section $s_0$ for some $s_1:G\to E_0$, then we would get a twisted set $H^1_\coc(G, {}_{s_1}(H/Z))$, where ${}_{s_1}(H/Z)$ is simply $H/Z$ with the ``twisted'' $G$-group structure obtained by conjugation via $s_1(G)\subseteq E_0$. We leave to the reader to check that the twisting bijection $H^1_\coc(G, {}_{s_1}(H/Z))\xrightarrow{\sim} H^1_\coc(G,H/Z)$, which is defined analogously to the classical group cohomology twist (cf.~\cite[I.5.3]{SerreCohGal}), commutes with the construction of the $\delta$ morphism. This implies in particular that the images of $\delta$ and $\delta_E$ are independent of the choice of the fixed section $s_0$. Moreover, if we change $E$ by another extension $E'$ inducing the same outer action, then we get a map $\delta_{E'}$, which we claim coincides with $\delta_E$. Indeed, we may write $[E']=\zeta\cdot [E]$ for some class $\zeta\in\ext(G,Z,E)$, meaning that $E'$ can be obtained by fiber products and quotients from $E$. Applying this construction simultaneously to the upper and lower lines of diagram \eqref{eq diag delta}, which preserves commutativity, we see that the corresponding map $\delta_{E'}$ sends $\alpha=[s]$ to \[(-(\delta(\alpha)+\zeta))\cdot [E']=(-\delta(\alpha))\cdot(-\zeta)\cdot [E']=(-\delta(\alpha))\cdot [E]=\delta_E(\alpha).\]

We have then the following result, which describes the set of (classes of) split extensions.

\begin{pro}
An extension in $\ext(G,H,E)$ is split if and only if it lies in the image of $\delta_E$. In particular, the set of split extensions is non-empty if and only if $E_0=E/Z$ is split as an extension of $G$ by $H/Z$.
\end{pro}

\begin{proof}
We already noticed that if the set of split extensions is non-empty, then $E_0$ is split. Assume then that $E_0$ is split and let us prove that there are split extensions in $\ext(G,H,E)$. Consider a section $s$ and the corresponding extension $E_s$. Applying simultaneously the action of $-\delta([s])=-[E_s]$ to the upper and lower lines of diagram \eqref{eq diag delta}, we get the same diagram for an extension $E'$ such that $[E']=(-\delta([s]))\cdot[E]$, but in which $E'_s$ is clearly split. If we denote by $s'$ the section $G\to E'$, since the composite $G\xrightarrow{s}E_0\to G$ is the identity, we have that $G\xrightarrow{s'} E'_s\to E'$ is a section of $E'$ and hence $E'$ is a split extension.

Note that this construction proves by the way that every class in the image of $\delta_E$ is split. Indeed, the extension $E'$ represents precisely $(-\delta([s]))\cdot[E]=\delta_E([s])$. Consider now a split extension in $\ext(G,H,E)$. Since the definition of $\delta_E$ does not depend on $E$, we may assume that our split extension is precisely $E$. Moreover, since the image of $\delta_E$ does not depend on the choice of the section $s:G\to E_0$, we may assume that this section comes from a section $G\to E$. We immediately see then that $\delta([s])$ is nothing but the trivial class $[Z\rtimes G]\in\ext(G,Z,E)$ and hence $[E]=(-\delta([s]))\cdot [E]=\delta_E([s])$.
\end{proof}

\begin{rem}
Assuming the existence of a split extension $E$, the arguments that precedes can be summarized as the existence of an exact sequence of pointed sets
\[H^1_\coc(G,H)\to H^1_\coc(G,H/Z)\xrightarrow{\delta} \ext(G,Z,E)\to \ext(G,H,E),\]
where the last map is defined as $\zeta\mapsto (-\zeta) \cdot[E]$ and the distinguished subset in $\ext(G,H,E)$ is precisely the set of split extensions. Details are left to the reader.
\end{rem}

\section{Application: a finiteness result}\label{section application}
We can now consider the special situation of algebraic groups over a finite field $\F$. In this context one could expect finiteness results such as the following one:

\begin{thm}
Let $G,H$ be algebraic groups over $\F$ and let
\[1\to H\to E\to G\to 1,\]
be an extension. Assume that $G$ is finite and \'etale. Then $\ext(G,H,E)$ is finite. 
\end{thm}

\begin{proof}
Thanks to Theorem \ref{main thm}, we may immediately assume that $H$ is abelian, hence it has a natural $G$-action and the set $\ext(G,H,E)$ is simply $\ext(G,H)$ for this particular action. We argue then as in \cite[Lem.~3.1]{MichelExt}: since $G$ is \'etale, we can consider the exact sequence (cf.~\cite[XVII, App.~I.3.1]{SGA3})
\[0\to H^2_0(G,H)\to\ext(G,H)\to H^1_{\text{\'et}}(G,H),\]
where $H^2_0(G,H)$ denotes the Hochschild cohomology group (defined in \textit{loc.~cit.}).
Again, since $G$ is \'etale, as a scheme it is just a finite union of spectra of finite fields. By \cite[III, Ex.~1.7 and Ex.~2.24]{Milne}, the finiteness of $H^1_{\text{\'et}}(G,H)$ follows then from the finiteness of $H^1(\F',H)$ for any finite field $\F'$, and this is well-known (cf.~\cite[III,\S4.3, Thm.~4]{SerreCohGal} and \cite{LangCorpsFinis}). We are reduced then to the finiteness of $H^2_0(G,H)$. Assume for now that $G$ is constant (i.e.~all of its points are defined over $\F$). Then by \cite[III.6, Prop.~4.2]{DG} we know that $H^2_0(G,H)$ is isomorphic to $H^2(G(\F),H(\F))$ in classic group cohomology. And since both $G(\F)$ and $H(\F)$ are finite groups, we get the finiteness in this case. If $G$ is not constant, there exists a finite extension $\F'/\F$ such that $G_{\F'}$ is constant and hence the set $\ext(G_{\F'},H_{\F'})$ is finite by the argument above. It suffices to prove then that there are finitely many $\F'/\F$-forms of a given extension of $G_{\F'}$ by $H_{\F'}$. Now these forms are in bijective correspondence with the cohomology group $H^1(\F'/\F,A)$ for the group $A:=Z_0^1(G_{\F'},H_{\F'})$ by \cite[III.1]{SerreCohGal} and Proposition \ref{prop aut ext}. But since $G_{\F'}$ is constant, $A$ is a subgroup of the group of functions $G(\F')\to H(\F')$, which is clearly finite. Since $\F'/\F$ is also finite, we get the finiteness of $H^1(\F'/\F,A)$, which concludes the proof.
\end{proof}

\end{document}